\documentclass[10pt,leqno]{amsart}
\usepackage{graphicx}
\usepackage{indentfirst,csquotes}

\usepackage{amssymb,amsthm,amsmath}
\usepackage{xcolor,paralist,hyperref,titlesec,fancyhdr,etoolbox}
\newtheorem{theorem}{Theorem}[]

\newtheorem{lemma}[theorem]{Lemma}

\titleformat{\section}[display]{\normalfont\huge\bfseries\centering}{\centering\chaptertitlename\thechapter}{10pt}{\Large}
\titlespacing*{\section}{0pt}{0ex}{0ex}

\hypersetup{ colorlinks=true, linkcolor=black, filecolor=black, urlcolor=black }

\usepackage{lipsum}

\begin{document}
\title{Geometric Properties of Planar and Spherical Interception Curves} 
\author[Y.N. Aliyev]{Yagub N. Aliyev}
\date{\today}
\address{ADA University, Ahmadbey Aghaoglu str. 61, Baku, 1008, Azerbaijan}
\email{yaliyev@ada.edu.az}
\maketitle

\let\thefootnote\relax
\footnotetext{MSC2020: 53A04, 49N75, 34A34, 91A24, 65D17, 34C05, 33E05, 97G60.} 

\begin{abstract}
In the paper, some geometric properties of the \textit{plane interception curve} defined by a nonlinear ordinary differential equation are discussed. Its parametric representation is used to find the limits of some triangle elements associated with the curve. These limits have some connections with the lemniscate constants $A,B$ and Gauss's constant $G$, which were used to compare with the classical pursuit curve. The analogous spherical geometry problem is solved using a spherical curve defined by the Gudermannian function. It is shown that the results agree with the angle-preserving property of Mercator and Stereographic projections. The Mercator and Stereographic projections also reveal the symmetry of this curve with respect to Spherical and Logarithmic Spirals. The geometric properties of the spherical curve are proved in two ways, analytically and using a lemma about spherical angles. A similar lemma for the planar case is also mentioned. The paper shows symmetry/asymmetry between the spherical and planar cases and the derivation of the properties of these curves as limiting cases of some plane and spherical geometry results.
\end{abstract} 

\bigskip

$\,$

$\,$
\textbf{1. Introduction}

The study of curves and their geometric properties is one of the oldest topics in mathematics. There are many interesting curves on a plane or sphere that are defined kinematically or geometrically and can be studied using the methods of calculus, differential equations, and differential geometry. For example, on a plane, one can mention Tractrix, Pursuit Curve, Lemniscate, Brachistochrone, Cycloid, Conchoid of Nicomedes, Limaçon of Pascal, etc., and on sphere Great Circle, Spherical Spiral, which is a special case of Loxodrome also known as Rhumb Line, etc. \cite{weis2}. The study of curves played a crucial role in the development of Mathematics and its applications. Various attempts to describe the curves led to the discovery of different coordinate systems, such as cartesian, polar, barycentric, etc. The study of geometric properties such as lengths, areas, angles, intersections, etc., led to the development of branches of Mathematics, such as Analytic Geometry, Calculus, Differential Geometry, Differential Equations, Topology, etc. In antiquity, the curves were used for the solution of construction problems such as trisection of an angle, squaring the circle, doubling the cube.
\, In the middle ages and renaissance, curves were used for the solution of various algebraic and geometric problems, for the description of trajectories of projectiles and planets, etc. See \cite{loria,loria2} for the history of curves. In modern times, many problems of science, technology, engineering, art, and mathematics require the study of various curves on different spaces and their geometric properties. Applications of curves include, but are not bound to, aeronautics, architecture, design, sea and air navigation, music, etc. With the dawn of computer graphics, it became easy to create and represent curves and use them in many aspects of life.

In Section 2 of the current paper, the curve, which is named the \textit{Interception Curve}, was studied. This curve can be obtained if one point moves uniformly along a straight line, and another point, initially one unit apart from both the line and the first point on this line, moves with the same speed so that it always stays on the line passing through the first point and the initial position of the second point. This plane curve appears in problems related to the interception of high-speed targets by beam rider missiles (hence the name \textit{Interception Curve}) \cite{elnan,vinh}. The curve was also studied in \cite{wilder,zbornik,bailey}. In \cite{kamke}, at Sect. 1.460 (polar coordinates) and Sect. 1.507 (cartesian coordinates), some methods were proposed to find a formula for this curve. In both cases, the obtained results were probably complicated enough to discourage the further study of this curve. In the present paper, the equation in cartesian coordinates was used  to prove some geometric properties of the curve. The cartesian equation was studied further to find a simple parametrization among other possible parametrizations. This parametrization helped us to prove some more geometric characteristics, which show its connections with Lemniscate constants and Gauss's constant~\cite{todd}. The use of mathematical software to perform experiments and computations and to draw accurate and colorful graphs was very important for the results of the current paper. For example, the connection between the limit of $|PQ|$ below and the lemniscate constants was suggested by Wolfram Alpha. Most of the graphs and pictures in the current paper were created using the website \url{https://www.geogebra.org/}\ (accessed on 17 February 2023).

Some of the curves appear naturally in pursuit-evasion problems. The Interception Curve can also be studied in this context. Comparing the well-known Pursuit Curve with the Interception Curve is done at the end of Section 2.

In Section 3, a similar question for a sphere was asked. In the spherical case, it is shown that the curve defined by the Gudermannian function has similar geometric properties, and these properties are proved first directly using analytic methods. At the end of Section 3, the connection of these characteristics with Mercator and Stereographic projections was discussed. These conformal projections also reveal the symmetry of this curve in connection with Spherical and Logarithmic Spirals.

In Section 4, it is shown how to use a lemma from spherical geometry to prove the mentioned geometric properties of the spherical curve indirectly. The planar case also accepts simple geometric methods, and these were mentioned at the end of Section 4. There is a remarkable symmetry/asymmetry between the planar and spherical cases (See Table 2, and going back and forward between these two cases was very helpful in writing more of these geometric characteristics.

\textbf{2. Planar Curve}\label{sec2}

Let us start with discussing the following question and setting the notation that will be used throughout the paper.

\textbf{Question 1.} Suppose that two points $P(x,y)$ and $Q$, initially at $O(0,0)$ and $A(1,0)$, respectively, move with constant and equal velocities so that $Q$ is on the line $x=1$, and $P$ is on the ray $OQ$. What curve is defined by the point $P$?

Similar questions were discussed in \cite{bailey,wilder}. Polar coordinates can be used to describe the resulting curve. Let us denote $r=|OP|$ and  $\angle AOQ=\theta$.

Since the speeds of the points $P$ and $Q$ are equal, the length of the curve $OP$ and the length of the line segment $AQ$, which is $\tan{\theta}$, are equal for each $\theta$. By using the well-known formula for the length of a curve $r=r(\theta)$, given in polar coordinates, we find that
\begin{equation}
\int_{0}^{\theta} \sqrt{r(t)^2+(r'(t))^2}dt=\tan{\theta}.
\end{equation}
By taking the derivative of both sides of (1) and simplifying, we obtain ODE
\begin{equation}
r(\theta)^2+(r'(\theta))^2=\frac{1}{\cos^4{\theta}},
\end{equation}
with initial condition $r(0)=0$. This nonlinear equation appears in problems related to the interception of high-speed targets by beam rider missiles \cite{elnan}. Therefore, in the current paper, the curve defined by (2) is called the \textit{Interception Curve}. See \cite{aliyev} for discussing a similar problem with $r(0)=1$.

One can find approximate solutions of (2) using series. If we substitute $\theta=0$ in (2), then we find that $r'(0)=\pm 1$. By taking the derivative of (2), and substituting $\theta=0$, we can find $r''(0)=0$. By continuing this process, we can find  $r'''(0)=\pm 1$, etc. Therefore, the two solutions are $r_1(\theta)=\theta+\frac{1}{6} \theta^3+\frac{7}{40} \theta^5+\frac{43}{720} \theta^7+O(\theta^9)$ and $r_2(\theta)=-r_1(\theta)$. Maple 2022  software was used to calculate these 4 nonzero terms of the series, which were then used in GeoGebra to draw Figure \ref{fig1}.

\begin{figure}
{\includegraphics[scale=.2]{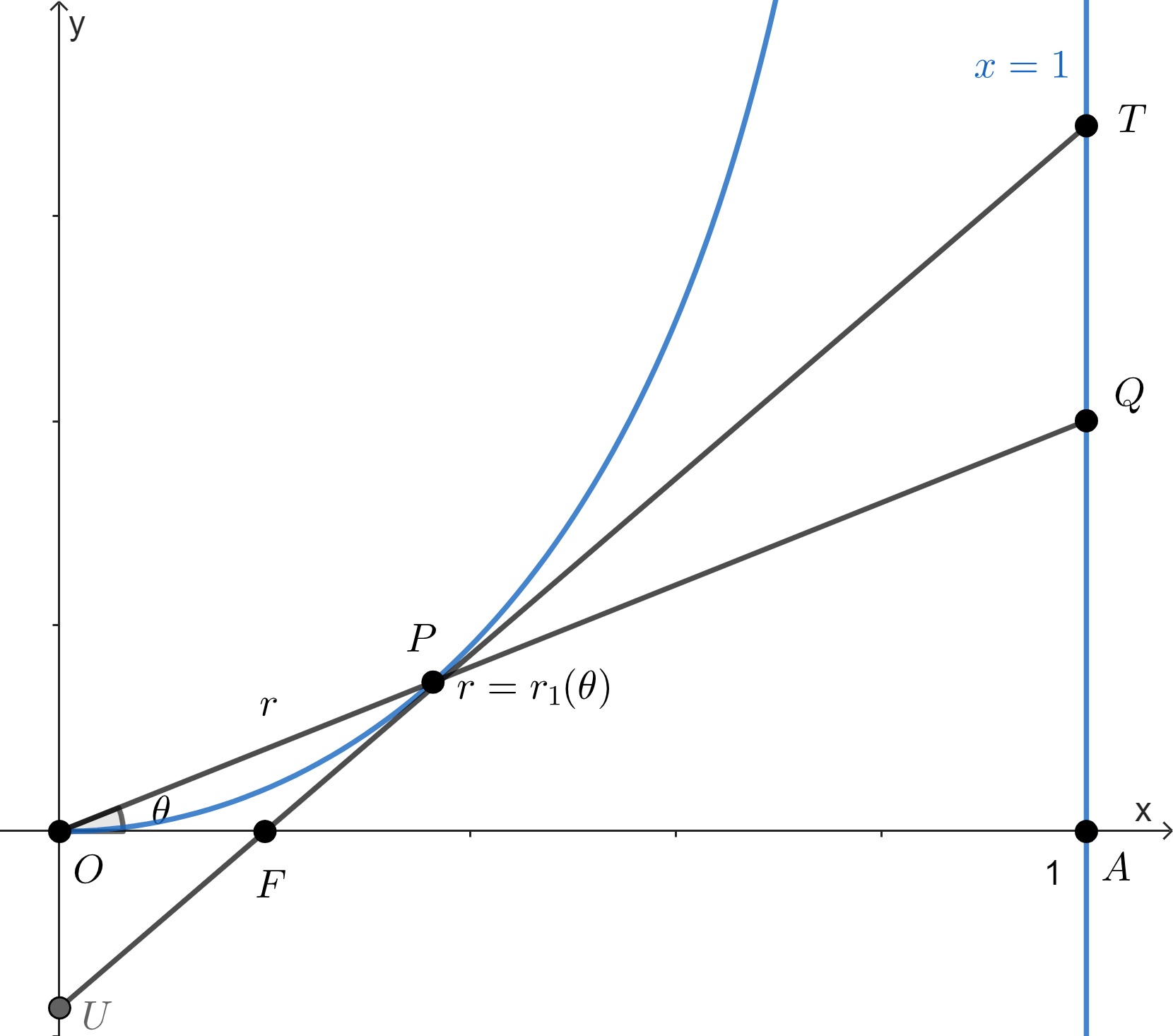}}
\caption{Interception Curve (blue) on a Plane. Its tangent line (black), the line containing its position vector (black), and the line $x=1$ (blue) are also shown.}
\label{fig1}
\end{figure}

Equation (2) can be solved by quadratures \cite{zbornik} (see also \cite{aliyev}). For example, one can follow the steps given in \cite{kamke} to find $r_1(\theta)$.

\begin{itemize}
\item Use substitution $r'=r\cot{u}$ to obtain $r\cos^2{\theta}=\pm a\sin{u}$. By differentiating and excluding $r$ and $r'$, we obtain $u'+2\tan{u}\tan{\theta}=1$ (\cite{kamke}, Sect. 1.460),

\item Use substitutions $\eta(\xi)=\tan{u}$, $\xi=\tan{\theta}$ to obtain Abel’s equation $(\xi^2+1)\eta '=(\eta^2+1)(1-2\xi \eta)$ (\cite{kamke}, Sect. 1.81),

\item Use substitution $\xi^4\eta(\xi)=(\xi^2+1)z+\xi^3$ to obtain one more Abel’s equation $\xi^7 z'+2(\xi^2+1) z^3+5\xi^3 z^2=0$ (\cite{kamke}, Sect. 1.151),

\item Use substitution $v=\frac{1}{z}$ to obtain $\xi^7 vv'=2(\xi^2+1)+5\xi^3 v$ (\cite{kamke}, Sect. 1.185),

\item Use substitution $\xi w=\xi^3 v+1$ to obtain linear equation $\frac{d\xi}{dw}-\frac{\xi w}{2(w^2+1)} +\frac{1}{2(w^2+1)} =0$, which can be solved by quadratures (\cite{kamke}, Sect. 1.185).
\end{itemize}

As one can understand from these steps, the final result will not be a simple expression. In particular, by using the substitution $w=\tan{s}$, the solution of the linear differential equation in the last step can be reduced to the integral $\int{\sec{s}ds}=2F(\frac{x}{2}| 2)+\textrm{const}$, where $F(x|m)$ is the elliptic integral of the first kind with parameter $m=k^2$ (cf. \cite{vinh}).

Note also that $r=r_1(\theta)$ and $x=1$ do not intersect for $0<\theta<\pi/2$, because otherwise, if the points $P$ and $Q$ coincided, then the length of the curve $OP$ would be equal to $AQ$, which is impossible. 

\textbf{Answer to Question 1.} Now, let us try using cartesian coordinates to find the solution of (2) as a parametric curve. First, note that in the cartesian coordinates, (1) can be written as
\begin{equation}
\int_{0}^{x} \sqrt{1+(y'(t))^2}dt=\frac{y}{x}.
\end{equation}
By taking the derivative of both sides of (3) and simplifying, we obtain
\begin{equation}
x^2\sqrt{1+(y'(x))^2}=y'x-y,
\end{equation}
which in particular agrees with $y(0)=0$. A parametrization of the solution can be written by solving (4) for $y$ to obtain $y=y'x-x^2\sqrt{(y')^2+1}=px-x^2\sqrt{p^2+1}$, where $y'=p\ge 0$. By noting that $dy=pdx$ and  $dy=pdx+xdp-2x\sqrt{p^2+1}dx-x^2\frac{p}{\sqrt{p^2+1}}dp$, a linear differential equation $\frac{dx}{dp}-\frac{xp}{2\sqrt{p^2+1}}=\frac{1}{2\sqrt{p^2+1}}$ is obtained. By solving this equation, we obtain the parametrization (cf. \cite{kamke}, Sect. 1.507, where the roles of $x$ and $y$ are interchanged) 
\begin{equation}
    \begin{cases}
x(p)=\frac{1}{\sqrt[4]{p^2+1}}\int_0^p\frac{dt}{2\sqrt[4]{t^2+1}},\\
y(p)=px(p)-(x(p))^2\sqrt{p^2+1}\ (p\ge 0).
    \end{cases}\,
\end{equation}
Of course, there are also other possible parametrizations of this curve (See \linebreak Appendix). This curve also has some interesting geometric properties. Some of these properties (Theorems 1 and 2) can be taken as its definition.

\begin{theorem}
Suppose that the tangent line of the curve (4) at the point $P$, intersects $x$-axis, $y$-axis, and the line $x=1$ at points $F$, $U$, and $T$, respectively (see Figure \ref{fig1}). Then
\begin{enumerate}
\item
$|UP|=|OU|+|TQ|$,
\item
$(1-x)\cdot |UP|=|TQ|$,
\item
$x\cdot |PT|=|TQ|$,
\item
$\sin{\angle QPT}=|OP|\cdot \sin^2{\angle TQP}$,
\end{enumerate}
where $x$ is the abscissa of the point $P(x,y)$.
\end{theorem}

\begin{proof}
The equation of the line $OQ$ is $\frac{Y-y}{X-x}=\frac{y}{x}$. By substituting $X=1$ in this equation, we can find the coordinates of point $Q\left(1,\frac{y}{x}\right)$. The equation of tangent line $UT$ is $\frac{Y-y}{X-x}=y'$. By substituting $Y=0$, $X=0$, and $X=1$ in this equation, we can find the coordinates of points $F$, $U$, and $T$, respectively:
$$
F\left(x-\frac{y}{y'},0\right), U\left(0,y-xy'\right), T\left(1,y+(1-x)y'\right).
$$
We can now find
$$
|UP|=x\sqrt{1+(y')^2}, |OU|=-y+xy',
$$
$$
|TQ|=y+y'-xy'-\frac{y}{x}=(1-x)\left(y'-\frac{y}{x}\right), |PT|=(1-x)\sqrt{1+(y')^2}.
$$
It remains only to put these formulas in the required equalities to check that each of the first three equalities is equivalent to (4). The fourth part follows from the previous equalities and sine theorem for $\triangle PQT$.
\qedhere
\end{proof}

It is also possible to prove these geometric properties using plane geometry methods, which we will discuss in Section 4. Using these plane geometry methods, one can also prove the following geometric property.

\begin{theorem}
The radius of the circle through $O$ and tangent to the line $UT$ at the point $P$ is equal to the radius of the circle through $O$ and tangent to the line $AT$ at the point $Q$ (see Figure \ref{fig1}).
\end{theorem}

\begin{proof}
The center of the circle through $O$ and tangent to the line $UT$ at the point $P$, is the intersection of perpendicular bisector $Y-\frac{y}{2}=-\frac{x}{y}\left(X-\frac{x}{2}\right)$ of line segment $OP$, and line  $Y-y=-\frac{1}{y'}\left(X-x\right)$ perpendicular to the line $PT$ at point $P$. We find that $X=x-\frac{(x^2+y^2)y'}{2(xy'-y)}$ and $Y=y+\frac{x^2+y^2}{2(xy'-y)}$. Therefore, the radius of this circle is $$\sqrt{\left(X-x\right)^2+\left(Y-y\right)^2}=\frac{x^2+y^2}{2(xy'-y)}\sqrt{1+(y')^2}.$$
Similarly, the center of the circle through $O$ and tangent to the line $AT$ at the point $Q$, is the intersection of perpendicular bisector $Y-\frac{y}{2x}=-\frac{x}{y}\left(X-\frac{1}{2}\right)$ of line segment $OQ$, and line  $Y=\frac{y}{x}$ perpendicular to the line $AT$ at point $Q$. Using this, we find that the radius of this circle is $1-X=\frac{x^2+y^2}{2x^2}$. It remains only to check that by (4), these radii are equal.
\qedhere
\end{proof}

The following result shows that there is a connection between the interception curve and the lemniscate functions.

\begin{theorem}
The length of the side $PQ$, and the difference of lengths of the other two sides of $\triangle PQT$ approach to the same limit $B^2$ as $x \rightarrow 1^-$: $$\lim_{x \rightarrow 1^-} |PQ| = \lim_{x \rightarrow 1^-} \left(|PT|-|TQ|\right)=\frac{\Gamma\left(\frac{3}{4}\right)^4}{2\pi}= B^2 \approx 0.3588850048,$$ where $B$ is the second lemniscate constant.
\end{theorem}

\begin{proof}
First, we calculate that $|PQ|=\sqrt{(1-x)^2+\left(\frac{(1-x)y}{x}\right)^2}$. Since $x(p)\rightarrow 1$ as $p \rightarrow +\infty$, $\lim_{x \rightarrow 1^-} |PQ| =\lim_{p \rightarrow +\infty} \left(\frac{(1-x)y}{x}\right)$. On the other hand, using the parametrization (5), we can write
$$
\frac{(1-x(p))y(p)}{x(p)}=(1-x(p))^2\sqrt{p^2+1}-\frac{1-x(p)}{\sqrt{p^2+1}+p}.
$$
The subtrahend $\frac{1-x(p)}{\sqrt{p^2+1}+p}$ in the last expression, approches to 0 as $p \rightarrow +\infty$. So, it remains only to find
$$
\lim_{p \rightarrow +\infty} (1-x(p))^2\sqrt{p^2+1}=\lim_{p \rightarrow +\infty} \left(\sqrt[4]{p^2+1}-\int_0^p\frac{dt}{2\sqrt[4]{t^2+1}}\right)^2.
$$
Since $\sqrt[4]{p^2+1}=1+\int_0^p\frac{tdt}{2\sqrt[4]{(t^2+1)^3}}$, 
$$
\lim_{p \rightarrow +\infty} (1-x(p))^2\sqrt{p^2+1}= \left(1-\int_0^{+\infty}\frac{\left(\sqrt{t^2+1}-t \right)dt}{2\sqrt[4]{(t^2+1)^3}}\right)^2.
$$
By substituting $t=\tan{s}$, where $0\le s<\frac{\pi}{2}$, in the improper integral, we obtain after simplifications
$$
\int_0^{+\infty}\frac{\left(\sqrt{t^2+1}-t \right)dt}{2\sqrt[4]{(t^2+1)^3}}=\int_0^{\frac{\pi}{2}}\frac{\sqrt{\cos{s}}ds}{2(1+\sin{s})}.
$$
It is known that (\cite{todd,cox})
$$
\int_0^{\frac{\pi}{2}}\frac{\sqrt{\cos{s}}ds}{2}=\int_0^{1}\frac{x^2dx}{\sqrt{1-x^4}}=\frac{\Gamma\left(\frac{3}{4}\right)^2}{\sqrt{2\pi}}= B.
$$
We only need to show that
$$
\int_0^{\frac{\pi}{2}}\frac{\sqrt{\cos{s}}ds}{2(1+\sin{s})}=1-B,
$$
which follows from
$$
\int_0^{\frac{\pi}{2}}\frac{\sqrt{\cos{s}}ds}{2(1+\sin{s})}+\int_0^{\frac{\pi}{2}}\frac{\sqrt{\cos{s}}ds}{2}=
$$
$$
\int_0^{\frac{\pi}{2}}\frac{\sqrt{\cos{s}}}{2}\left(\frac{1}{1+\sin{s}}+1\right)ds=\frac{\sqrt{\cos{s}}}{2}\left(\frac{4\sin{\frac{s}{2}}}{\sin{\frac{s}{2}}+\cos{\frac{s}{2}}}-2\right)\biggm\lvert_0^{\frac{\pi}{2}}=1.
$$
Similarly,
$$
|PT|-|TQ|=(1-x)\left(\sqrt{1+(y')^2}-y'+\frac{y}{x}\right).
$$
Therefore,
$$
 \lim_{x \rightarrow 1^-} \left(|PT|-|TQ|\right)=\lim_{p \rightarrow +\infty} (1-x(p))\left(\sqrt{p^2+1}-p+\frac{y(p)}{x(p)}\right)=
$$
$$
=\lim_{p \rightarrow +\infty} \frac{1-x(p)}{\left(\sqrt{p^2+1}+p\right)}+\lim_{p \rightarrow +\infty}\left(\frac{(1-x(p))y(p)}{x(p)}\right)=0+B^2=B^2.
$$
\qedhere
\end{proof}
Note that $B^2=\frac{1}{4G^2}$, where $G$ is Gauss's constant defined by arithmetic–geometric mean \cite{cox,todd}.

There are other strategies for making the point $P$ follow the point $Q$ so that $\lim_{x \rightarrow 1^-} |PQ|$ is smaller. However, it is possible to show that there is no optimal strategy. Consider, for example, this pursuit evasion game. Suppose that there is a barrier along the line $x=1$ and an evader is at point $(1,0)$. A pursuer at the origin starts to chase with a constant speed, and the evader moves with the same speed along the line $x=1$. One strategy for the pursuer can be to move always towards the evader. This gives rise to a trajectory for the pursuer, commonly known as \textit{pursuit curve} \cite{savelov} (page 256). However, the pursuer can also move towards a future position of the evader instead of the evader's current position. One of the many possible ways for the pursuer to do this is to stay on the line connecting the origin with the evader. The optimal strategy for the evader is to run along the barrier $x=1$ by never changing its direction. However, there is no optimal trajectory for the pursuer, so any candidate for the optimal curve can be improved to make $\lim_{x \rightarrow 1^-} |PQ|$ smaller, by replacing one of its curved parts with a line segment and, therefore, shortening that part. Because of the presence of the barrier $x=1$, the curved part should always exist. Note that any straight line through the origin intersects $x=1$, except the line $x=0$, which obviously can not be the optimal curve. One can pose a similar question for a sphere, where the barrier and the evader are on a great circle of the sphere, and the pursuer is at one of the poles of the great circle. This will be done in the following section.

Consider also this trajectory: first $P$ moves with constant speed from $O$ to the point with coordinates $(1,k)$, along the line $y=kx$, where $k\ge 0$, $0\le x\le 1$, while $Q$ moves with the same speed from $A$ to the point with coordinates $(1,\sqrt{k^2+1})$, after which $P$ follows $Q$ along the line $x=1$ and the distance $|PQ|=\sqrt{k^2+1}-k$ does not change. The value of $\sqrt{k^2+1}-k=\frac{1}{\sqrt{k^2+1}+k}$ can be made arbitrarily close to zero. On the other hand, for any chosen curve $\lim_{x \rightarrow 1^-} |PQ|=0$ is not possible. Indeed, $|PQ|\ge |(PQ)_y|$, where $(PQ)_y$ is the projection of $PQ$ onto $y$-axis, and $|(PQ)_y|$ is a non-decreasing function of time (or $x$). Therefore, for any given curve with a particular positive value of $\lim_{x \rightarrow 1^-} |PQ|$, there is always a more efficient curve making the limit even smaller. In view of these observations, it would be interesting to compare the curve in Figure \ref{fig1} with the mentioned trajectory of $P$ chasing $Q$, the pursuit curve.

\begin{figure}
{\includegraphics[scale=.6]{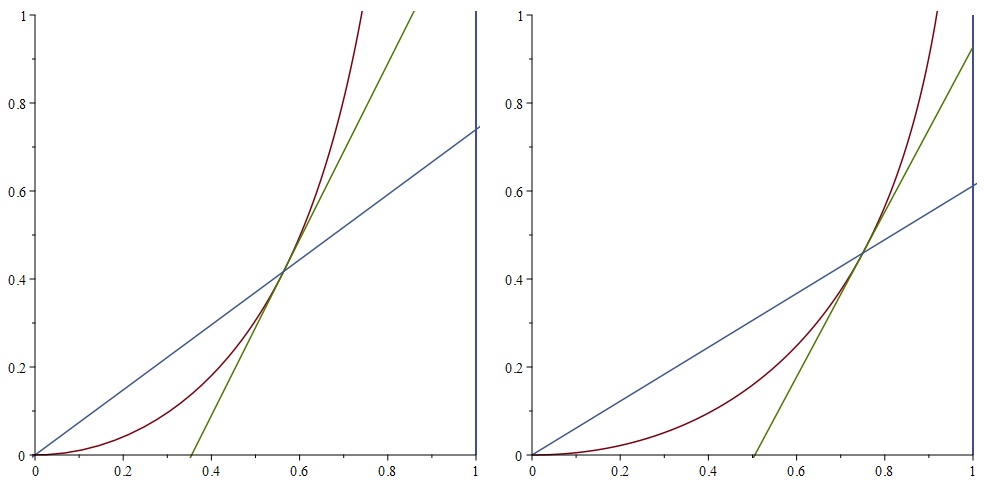}}
\caption{Comparison of Interception (left, red) and Pursuit (right, red) Curves. Created using parametrization (5), and Bouguer's formula. The tangent lines (green), the lines containing the position vectors (thin blue) of the curves, and the lines $x=1$ (thick blue) are also shown.}
\label{fig2}
\end{figure}

Figure \ref{fig2} was created using Maple 2022. The software Maple 2022 was also used to simplify limits and integrals in the paper.

Suppose now that $P$ starts to move from $O$ with constant speed along such a curve that the tangent line of the curve at $P$ always passes through $Q$, which moves as before with the same constant speed from $A$ upwards, along the line $x=1$. The obtained curve is known as \textit{Pursuit Curve} and defined by the differential equation $\sqrt{1+(y')^2}=(1-x)y'',$ with initial conditions $y(0)=y'(0)=0$. See, for example, \cite{savelov} (page 256), \cite{loria2} (page 241 of 1911 edition, page 607 of 1902 edition) for the solution found by P. Bouguer in 1732: $$y=\frac{1}{4}(1 - x)^2 - \frac{1}{2}\ln{(1 - x)} - \frac{1}{4}.$$ One can check that for this curve $|PQ|=\frac{1}{2}+\frac{(1-x)^2}{2}$, and, therefore, $\lim_{x \rightarrow 1^-} |PQ|=\frac{1}{2}>B^2$, which means that the interception curve overperforms the pursuit curve in this regard.

\textbf{3. Spherical Curve}\label{sec3}

We will now study analogous questions for spherical geometry.

\textbf{Question 2.} Suppose that two points $P$ and $Q$, initially at $B(0,0,1)$ and $A(1,0,0)$, respectively, move with constant and equal velocity so that $Q$ is on the great circle $z=0$, $x^2+y^2=1$ of sphere $x^2+y^2+z^2=1$ with center $O(0,0,0)$, and $P$ is on the great circle through $B$ and $Q$ of the sphere. What curve is defined by the point $P$? (see Figure \ref{fig3})

\textbf{Answer.} We will use spherical coordinates to describe the resulting \textit{interception curve}. Denote $\angle AOQ=\theta$ and $\angle POB=\phi$. Since $\rho=|OP|=1$, for the coordinates of point $P(x,y,z)$, we can write $x=\cos{\theta}\sin{\phi}$, $y=\sin{\theta}\sin{\phi}$, and $z=\cos{\phi}$, where we assume that $\phi=\phi(\theta)$ is a function of $\theta$. So, $x'_\theta=-\sin{\theta}\sin{\phi} +\phi'\cos{\theta}\cos{\phi}$, $y'_\theta=\cos{\theta}\sin{\phi}+\phi'\sin{\theta}\cos{\phi}$ and $z'_\theta=-\phi'\sin{\phi}$.  Since the points $P$ and $Q$ travel equal distances, by using the well-known formula from vector calculus for the length of a parametrically defined curve, we can write
\begin{equation}
\int_{0}^{\theta} \sqrt{(x'(t))^2+(y'(t))^2+(z'(t))^2}dt=\theta,
\end{equation}
where $\theta$ on the right hand side of the last equality is the distance traveled by $Q$. By taking the derivative of both sides of (6) with respect to $\theta$ and simplifying using the formula for $x'_\theta$, $y'_\theta$, and $z'_\theta$, we obtain
$\sqrt{\sin^2{\phi}+(\phi'(\theta))^2}=1,$
with initial condition $\phi(0)=0$. By solving this ODE, we obtain
\begin{equation}
\phi=2\tan^{-1}{e^\theta}-\frac{\pi}{2}.
\end{equation}
Note that (7), which can also be written as $\phi=\tan^{-1}{\sinh{\theta}}$, is sometimes called Gudermannian function gd\smallskip$(x)$. See for example, \cite{olver}, Sect. 4.23(viii) and \cite{zwil}, Sect. 6.12. Gudermannian function gd\smallskip$(x)$ is the vertical component of Mercator projection \cite{weis}, and we will see below that this connection with the Mercator projection is not a coincidence.

\begin{figure}
{\includegraphics[scale=.65]{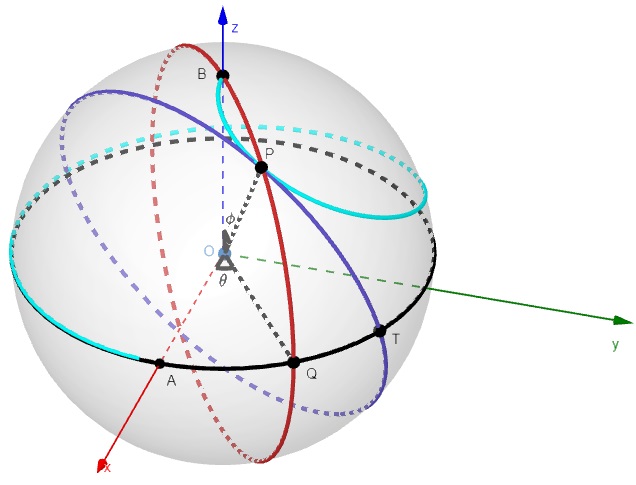}}
\caption{Interception Curve (light blue) on a Unit Sphere. Its tangent great circle (dark blue), meridian (red) and equator (black) great circles are also shown.}
\label{fig3}
\end{figure}

\begin{theorem}
$\lim_{\theta \rightarrow \infty} |PQ|=0.$
\end{theorem}
\begin{proof} One can observe that $$\lim_{\theta \rightarrow \infty} \phi(\theta)=\lim_{\theta \rightarrow \infty} \left(2\tan^{-1}{e^\theta}-\frac{\pi}{2}\right)=2\cdot\frac{\pi}{2}-\frac{\pi}{2}=\frac{\pi}{2},$$ and, therefore, the distance between the points $P(1,\theta,\phi)$ and $Q(1, \theta, \frac{\pi}{2})$ approaches zero as the curve winds around the sphere.
\qedhere
\end{proof}
This curve has other geometric properties. In the following, $\widehat{XY}$ means the spherical distance between points $X$ and $Y$ on a sphere. For a unit sphere with center $O$, $\widehat{XY}=\angle{XOY}$.

\begin{theorem}
If a great circle tangent to the curve $\phi=2\tan^{-1}{e^\theta}-\frac{\pi}{2}$ at point $P$ of unit sphere intersects the great circle on the plane $xOy$ at point $T$ (see Figure \ref{fig3}), then
\begin{enumerate}
\item
the sum of the lengths of the arcs $PT$ and $TQ$ is not dependent on $\theta$, and $\widehat{PT}+ \widehat{TQ}=\frac{\pi}{2}$,
\item
as $\theta$ increases, $ \widehat{TQ}$ increases, $\widehat{PT}$ decreases, and both approach to $\frac{\pi}{4}$, as $\theta \rightarrow \infty$.
\item
spherical angle $\angle QPT$ is equal to $\widehat{BP}$,
\item
spherical angle $\angle BPT$ is equal to $\widehat{PQ}+\frac{\pi}{2}$.
\end{enumerate}
\end{theorem}

\begin{proof}
We continue to use the notations introduced in Question 2 and its answer. Since $\phi'(\theta)=\cos{\phi}$, we can write for tangent vector $\vec{v}=(x'_\theta,y'_\theta,z'_\theta)$ of the curve (7) the following
$$
\vec{v}=(-\sin{\theta}\sin{\phi} +\cos{\theta}\cos^2{\phi},\cos{\theta}\sin{\phi}+\sin{\theta}\cos^2{\phi},-\sin{\phi}\cos{\phi}).
$$
We can find the normal vector of the plane containing the great circle through $P$ and $T$, as $\vec{n_1}=\vec{v}\times \vec{r}$, where $\vec{r}=(x,y,z)$:
$$
\vec{n_1}=(\cos{\theta}\sin{\phi}\cos{\phi}+\sin{\theta}\cos{\phi},\sin{\theta}\sin{\phi}\cos{\phi}-\cos{\theta}\cos{\phi},-\sin^2{\phi}).
$$
Then the equation of this plane is
$$
(\cos{\theta}\sin{\phi}\cos{\phi}+\sin{\theta}\cos{\phi})x+(\sin{\theta}\sin{\phi}\cos{\phi}-\cos{\theta}\cos{\phi})y-(\sin^2{\phi})z=0.
$$
By putting $z=0$ in this equation, we can find the only unknown coordinate $\Theta$ of the point $T(1,\Theta,\frac{\pi}{2})$ in spherical coordinates:
$$
\tan{\Theta}=\frac{y}{x}=\frac{\cos{\theta}\sin{\phi}\cos{\phi}+\sin{\theta}\cos{\phi}}{-\sin{\theta}\sin{\phi}\cos{\phi}+\cos{\theta}\cos{\phi}}=\frac{\tan{\theta}+\sin{\phi}}{1-\tan{\theta}\sin{\phi}}.
$$
By (7), $\sin{\phi}=\tanh{\theta}$. Therefore, $$\tan{\Theta}=\tan{(\theta+\tan^{-1}{\tanh{\theta}})},$$ which means that $\Theta=\theta+\tan^{-1}{\tanh{\theta}}$ and consequently $\widehat{QT}=\tan^{-1}{\tanh{\theta}}$. Using Spherical Pythagorean Theorem (\cite{whittlesey}, page 112) for $\triangle PQT$, we obtain $$\cos{\widehat{PT}}=\cos{\widehat{PQ}}\cos{\widehat{QT}}=\sin{\phi}\cos{\tan^{-1}{\tanh{\theta}}}=\tanh{\theta}\frac{\cosh{\theta}}{\cosh{2\theta}}=\frac{\sinh{\theta}}{\cosh{2\theta}}.$$
Therefore, $\tan{\widehat{PT}}=\coth{\theta}$, and consequently $\tan{\widehat{PT}}\tan{\widehat{QT}}=\coth{\theta}\tanh{\theta}=1$. Since $\widehat{PT},\widehat{TQ}<\frac{\pi}{2}$, this proves that $\widehat{PT}+ \widehat{TQ}=\frac{\pi}{2}$. Using the formula $\widehat{PT}=\tan^{-1}{\coth{\theta}}$ and $\widehat{QT}=\tan^{-1}{\tanh{\theta}}$ we can now show that as $\theta$ increases, $ \widehat{TQ}$ increases, but $\widehat{PT}$ decreases, and both approach to $\frac{\pi}{4}$. Using the formula for a right angled spherical triange (\cite{whittlesey}, page 127, 4.35), we obtain
$$
\tan{\angle QPT}=\frac{\tan{\widehat{QT}}}{\sin{\widehat{PT}}}=\frac{\tanh{\theta}}{\sin{\left(\frac{\pi}{2}-\phi\right)}}=\frac{\sin{\phi}}{\cos{\phi}}=\tan{\phi}.
$$
Therefore, $\angle QPT=\widehat{BP}$. Similarly, $\angle BPT=\widehat{PQ}+\frac{\pi}{2}$.
\qedhere
\end{proof}
Using similar analytic methods, one can prove the following property of the curve (7) for tangent small circles. It is also possible to prove this theorem using Lemma 2.
\begin{theorem}
If a small circle of the unit sphere passes through point $B$ and is tangent to the curve (7) at point $P$ (see Figure \ref{fig3}), then its spherical radius $R$ satisfies $\tan^2{R}=\frac{1}{4}\sec^4{\frac{1}{2}\widehat{BP}}$.
\end{theorem}

We will now discuss the mentioned connection with the Mercator projection. Let us do a projection of the sphere onto the cylinder $x^2+y^2=1$ using the formula $x=\theta$ and $y=\ln{\tan{\left(\frac{\pi}{4}+\frac{\varphi}{2}\right)}}$ (see
\cite{pearson}, p. 191). Note that here $x$ and $y$ are new coordinates over the cylinder \cite{weis}. Note also the difference between the spherical coordinate $\phi$ and latitude $\varphi =\frac{\pi}{2}-\phi$. Using (7) and the last two equalities, we can write 
$$y=\ln{\left(\frac{\pi}{4}+\frac{\frac{\pi}{2}-\left(2\tan^{-1}{e^\theta}-\frac{\pi}{2} \right)}{2}\right)}=\ln{\coth{\frac{\theta}{2}}}.$$
Since $x=\theta$ on the cylinder, we determine that Mercator projection of the curve (7) is $y=\ln{\coth{\frac{x}{2}}}$, which is an inverse function of itself. This means that after the application of Mercator projection, the curve (7) has a \textit{symmetry} with respect to the line $y=x$ (See Figure~\ref{fig4}). Note that $y=x$ is a Helix on the cylinder, and it is the projection of a Spherical Spiral (a special case of Loxodrom or Rhumb Line \cite{weis2}) intersecting the meridians $\phi=\textrm{const}$ of the sphere at constant angle $\frac{\pi}{4}$. Since Mercator projection is conformal (angle-preserving), the slope $y'=-\frac{1}{\sinh{x}}$ on the cylinder and $\cot{\angle BPT}$ on the sphere should be equal. Indeed, by Theorem 5, $$\cot{\angle BPT}=\cot{\left(\pi-\phi \right)}=-\cot{\left(\phi \right)}$$ $$=-\cot{\left( 2\tan^{-1}{e^\theta}-\frac{\pi}{2}\right)}=\tan{\left( 2\tan^{-1}{e^\theta}\right)}=\frac{2e^\theta}{1-e^{2\theta}}=-\frac{1}{\sinh{\theta}}.$$ 
\vspace{-12pt}
\begin{figure}
{\includegraphics[scale=.5]{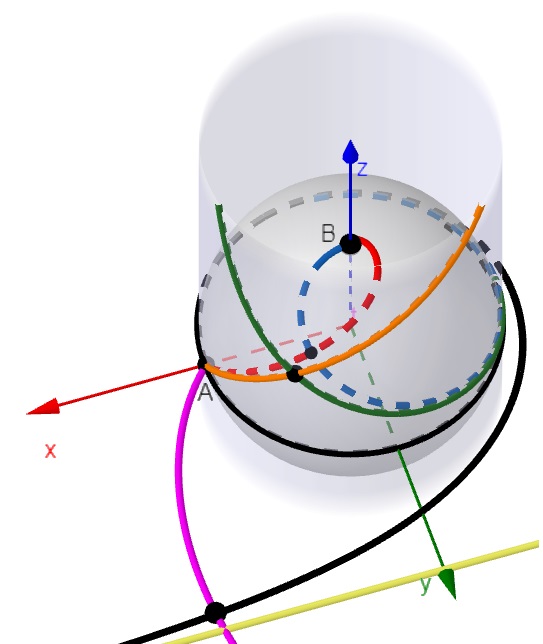}}
\caption{Mercator and Stereographic projections of the Spherical Interception Curve (blue) and Spherical Spiral (red). The other colors are explained in the text and in Table 1.}
\label{fig4}
\end{figure}

Note that in the literature the expression \textit{Spherical Spiral} is sometimes used for different curves. In Figure \ref{fig4}, we used parametrization $x=\frac{\cos{\theta}}{\cosh{\theta}}, y=\frac{\sin{\theta}}{\cosh{\theta}}, z=\tanh{\theta}$ for this curve (red curve) (\cite{gray}, Lemma 8.27), the parametrization $x=\cos{\theta}, y=\sin{\theta}, z=\theta$ for $y=x$ (orange curve), which is Mercator projection of the previous curve onto the cylinder, the parametrization $x=\cos{\theta}, y=\sin{\theta}, z=\ln{\coth{\frac{\theta}{2}}}$ for $y=\ln{\coth{\frac{x}{2}}}$ (green curve), which is Mercator projection of the interception curve (7) (blue curve). Note also that in Figure  \ref{fig4}, Stereographic projection of (7) onto $xy$ plane, with respect to the pole $B$, is a plane spiral curve (black curve), which in polar coordinates can be written as $r=\coth{\frac{\theta}{2}}$. Indeed, by (7),
$$r=\frac{1}{\tan{\frac{\phi}{2}}}=\frac{1}{\tan{\left(\tan^{-1}{e^\theta}-\frac{\pi}{4} \right)}}=\coth{\frac{\theta}{2}}.$$
The line $y=2$ (yellow line) on the $xy$ plane is the asymptote of this curve. It is known that the Stereographic projection of the rhumb line is a logarithmic spiral $r=e^{\theta}$ (purple curve) (\cite{gray}, Lemma 8.27). In particular, since the curves $y=x$ and $y=\ln{\coth{\frac{x}{2}}}$ intersect at a right angle, and Stereographic and Mercator projections are both conformal, the curves on the sphere, on the cylinder, and the plane in Figure \ref{fig4}, intersect at right angles at the indicated points. See Table 1 for the list of corresponding curves in Figure \ref{fig4}.

\begin{table}
\caption{Corresponding curves on the sphere, on the cylinder and on the plane in Figure \ref{fig4}.\label{Table 2}}

\begin{tabular}{|c|c|c|}
 \hline
\textbf{Sphere \boldmath{$x^2+y^2+z^2=1$}}	\\
 \hline
$\phi=2\tan^{-1}{e^\theta}-\frac{\pi}{2}$ Interception Curve (blue)\\
 \hline
$\left(\frac{\cos{\theta}}{\cosh{\theta}}, \frac{\sin{\theta}}{\cosh{\theta}}, \tanh{\theta}\right)$\\
 \hline
 \hline
 \textbf{Cylinder \boldmath{$x^2+y^2=1$}}	\  \textbf{(Mercator Projection)}\\
 \hline
$y=\ln{\coth{\frac{x}{2}}}$\newline (green)\\
 \hline$y=x$ \newline Helix (orange)\\
 \hline
 \hline
 \textbf{Plane \boldmath{$z=0$}} \  \textbf{(Stereographic Projection)}\\
 \hline
 $r=\coth{\frac{\theta}{2}}$ \newline (black)\\
 \hline
$r=e^{\theta}$ \newline Logarithmic Spiral (purple)\\
 \hline

\end{tabular}
\end{table}

\textbf{4. Alternative Methods for Spherical and Planar Curves}\label{sec4}

We can prove Theorem 5 also using the following lemma, which is interesting on its~own.
\begin{lemma} On a unit sphere with center $O$, a great circle and one of its poles $B$ is drawn. Two great circles through $B$ intersect the first great circle at $Q_1$ and $Q_2$, so that $\widehat{Q_1Q_2}<\frac{\pi}{2}$. Another great circle intersects arcs $BQ_1$ and $BQ_2$ at points $P_1$ and $P_2$ so that $\widehat{P_1P_2}=\widehat{Q_1Q_2}$. This great circle also intersects the first great circle at $T$ (See Figure \ref{fig5}). Then $$\widehat{P_1T}+ \widehat{TQ_2}=\widehat{P_2T}+ \widehat{TQ_1}=\frac{\pi}{2}.$$
\end{lemma}
\vspace{-12pt}
\begin{figure}
{\includegraphics[scale=0.6]{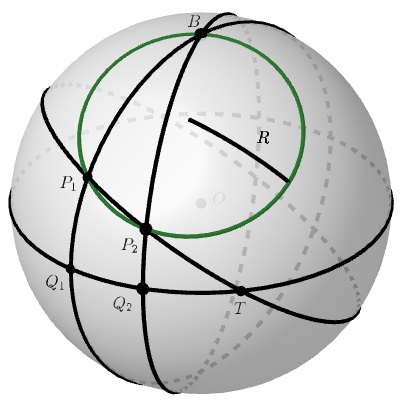}}
\caption{Lemma 1 for Spherical Case. The green small circle will be used in Lemma 2.\label{fig5}}
\end{figure}
\begin{proof} Let us denote $\widehat{P_1P_2}=\widehat{Q_1Q_2}=x$. Since angles at $Q_1$ and $Q_2$ are right angles, by spherical sine theorem for $\triangle BQ_1Q_2$, the angle between the two great circles at vertex $B$ is also $x$ (\cite{whittlesey}, page 30). Denote also $\widehat{P_1Q_1}=a$, $\widehat{P_2Q_2}=b$ (See Figure \ref{fig6}). Then $\widehat{BP_1}=\frac{\pi}{2}-a$ and $\widehat{BP_2}=\frac{\pi}{2}-b$. By spherical sine theorem for $\triangle BP_1P_2$ (\cite{whittlesey}, page 115),
$$
\frac{\sin{B}}{\sin{\widehat{P_1P_2}}}=\frac{\sin{P_1}}{\sin{\widehat{BP_2}}}=\frac{\sin{P_2}}{\sin{\widehat{BP_1}}}.
$$
Since $\sin{B}=\sin{\widehat{P_1P_2}}=\sin{x}$,
\begin{equation}
\sin{\angle BP_1P_2}=\sin{\widehat{BP_2}}=\cos{b}, \sin{\angle BP_2P_1}\sin{\widehat{BP_1}}=\cos{a}.
\end{equation}
Take another point $P_3$ on the arc $P_1Q_1$, such that $\widehat{P_2P_3}=x$. Such a point $P_3$ should exist, because $\widehat{P_2Q_1}>x$. Denote the midpoint of $\widehat{P_1P_3}$ by $E$ and let $\widehat{Q_1P_3}=c$. Then $\widehat{P_1E}=\widehat{EP_3}=0.5(a-c)$. Since spherical $\triangle P_1P_2P_3$ is isosceles, $P_2E$ is its angle bisector and $P_2E$ is perpendicular to $P_1P_3$. The three angles of the quadrilateral $EQ_1Q_2P_2$ are right angles. Therefore, the fourth $\angle EP_2Q_2$ is obtuse and $\angle EP_2B$ is acute. Since $\angle P_1P_2B<\angle EP_2B$, $\angle P_1P_2B$ is acute, too. Similarly, $\angle BP_3P_2=\angle P_3P_1P_2$ is acute. Consequently, $\angle BP_1P_2$ is obtuse. So, $\angle BP_1P_2$ and $\angle BP_2P_3$ are obtuse, $\angle BP_2P_1$ and $\angle BP_3P_2$ are acute. Therefore, by (8), $\angle BP_1P_2=\frac{\pi}{2}+b$ and $\angle BP_2P_1=\frac{\pi}{2}-a$. Similarly,  $\angle BP_2P_3=\frac{\pi}{2}+c$ and $\angle BP_3P_2=\frac{\pi}{2}-b$. So, $\angle PP_2P_3=\angle BP_2P_3-\angle BP_2P_1=c+a$. Therefore, $\angle P_1P_2E=\frac{c+a}{2}$. 
Denote the intersection of the great circles $EP_2$ and $Q_1Q_2$ by $D$. Denote also $EP_2=w$, $P_2T=z$, $Q_2T=y$, $TD=z_1$. Then $P_2D=\frac{\pi}{2}-w$ and $x+y+z_1=\frac{\pi}{2}$. It remains only to show that $z=z_1$. This follows from the fact that spherical $\triangle P_2TD$ is isosceles. Indeed,  $\angle D=\widehat{Q_1E}=\frac{a+c}{2}$ and $\angle TP_2D=\angle P_1P_2E=\frac{a+c}{2}$. Consequently, $\angle D=\angle TP_2D$ or $z=z_1$. Finally
$$\widehat{P_1T}+ \widehat{TQ_2}=\widehat{P_2T}+ \widehat{TQ_1}=x+y+z=\frac{\pi}{2}.$$
\qedhere
\end{proof}
\vspace{-12pt}
\begin{figure}
{\includegraphics[scale=.9]{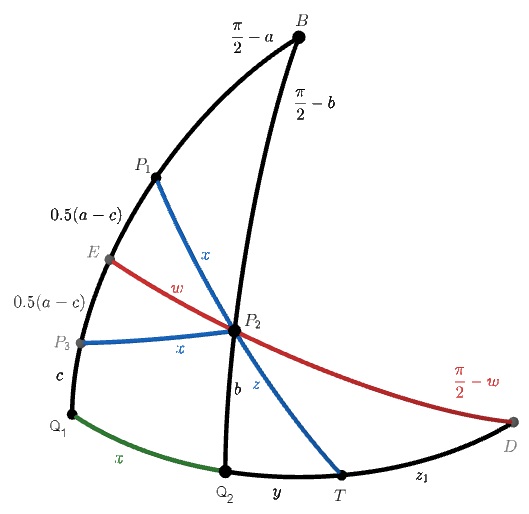}}
\caption{Proof of Lemma 1.\label{fig6}}
\end{figure}

For the proof of Theorem 5, one can observe that infinitesimal line element or length element of the curve $\phi=2\tan^{-1}{e^\theta}-\frac{\pi}{2}$ satisfies the conditions of Lemma 1
 and, therefore, one just needs to pass to limit $x\rightarrow 0$ ($Q_1$ approaches $Q_2$) which implies $\widehat{PT}+ \widehat{TQ}=y+z=\frac{\pi}{2}$. Since $\lim_{\theta \rightarrow \infty} |PQ|=0$, $y$ and $z$ both approach $\frac{\pi}{4}$. Passing to limit $x\rightarrow 0$,  also implies $a\rightarrow b$. Therefore,  $\angle QPT=\lim_{x \rightarrow 0} \angle P_1P_2B=\lim_{x \rightarrow 0} \left(\frac{\pi}{2}-a\right)=\frac{\pi}{2}-b=\widehat{BP}$. Similarly, $\angle BPT=\lim_{x \rightarrow 0} \angle BP_2T=\lim_{x \rightarrow 0}  \left(\pi-\angle P_1P_2B\right)=\lim_{x \rightarrow 0}  \left(\frac{\pi}{2}+a\right)=\frac{\pi}{2}+b=\widehat{PQ}+\frac{\pi}{2}$.

We can prove Theorem 6 using the following lemma, which is a direct consequence of the formula for the radius of the circumscribed circle of a spherical triangle (see, e.g., \cite{odani}, Formula (4)).

\begin{lemma} Under conditions of Lemma 1, let a small circle of the unit sphere be circumscribed about $\triangle BP_1P_2$ (See Figure \ref{fig5}). If radius of the small circle is $R$, then $$\tan^2{R}=\frac{1}{4}\sec^4{\frac{1}{2}\widehat{P_1P_2}}\sec^4{\frac{1}{2}\widehat{BP_1}}\sec^4{\frac{1}{2}\widehat{BP_2}}.$$
\end{lemma}

It is possible to give similar simple proofs for the equalities in the planar case, too. The equalities in Theorems 1 and 2 can be proved as a limiting case of the following lemma when $Q_1$ approaches $Q_2$.

\begin{lemma} A line passing through a point $O$ and points $Q_1$ and $Q_2$ on a parallel line are drawn. A line intersects these lines at $U$ and $T$, and the segments $OQ_1$ and $OQ_2$ at points $P_1$ and $P_2$, respectively, so that $|P_1P_2|=|Q_1Q_2|$ (see Figure \ref{fig7}). Then
\begin{enumerate}
\item
$|OU|+|TQ_2|=|UP_1|$ and $|OU|+|TQ_1|=|UP_2|$,
\item
the radii of circles through the points $O,P_1,P_2$ and $O,Q_1,Q_2$ are equal,
\item
if the distance between the parallel lines is 1, and the distances from the points $P_1$ and $P_2$ to the line $OU$ are $x_1$ and $x_2$, respectively, then $$\sin{\angle Q_1P_1T}\cdot \sin{\angle Q_2P_2T}=|OP_1|\cdot|OP_2| \cdot\sin^2{\angle TQ_1P_1}\cdot\sin^2{\angle TQ_2P_2},$$ $$x_1x_2\cdot|P_1T|\cdot |P_2T|=|TQ_1|\cdot |TQ_2|,$$ $$(1-x_1)(1-x_2)\cdot|UP_1|\cdot |UP_2|=|TQ_1|\cdot |TQ_2|.$$
\end{enumerate}
\end{lemma}

\begin{figure}
{\includegraphics[scale=.5]{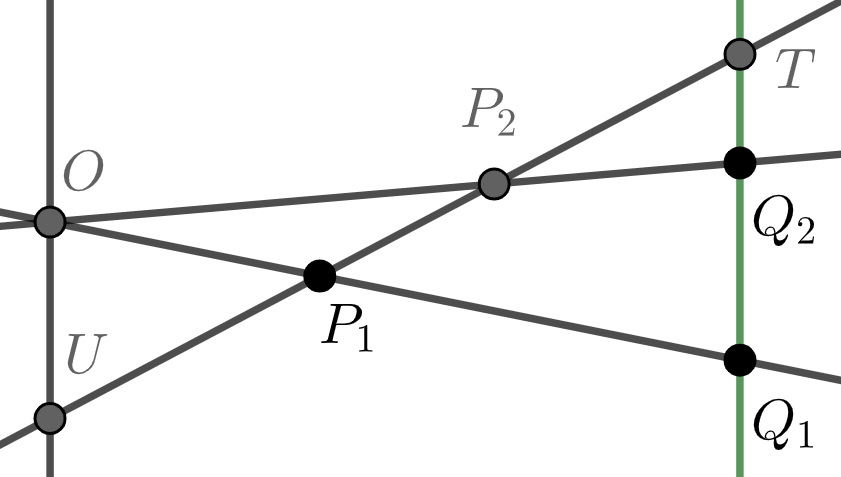}}
\caption{Lemma 3 for Planar Case.}
\label{fig7}
\end{figure}

See Table 2 for a complete picture of corresponding lemmas.
\begin{table}
\caption{Corresponding theorems in planar and spherical cases, and lemmata in Section 4 that can be used to prove these results.\label{TableA1}}

\begin{center}
\begin{tabular}{ |c|c|c|c|c|} 
 \hline
\textbf{Plane}& \textbf{Sphere}	 & \textbf{Plane}& \textbf{Sphere}\\
 \hline
Theorem 1 &Theorem 5 &Lemma 3, 1. \& 3.&Lemma 1\\
Theorem 2 &Theorem 6&Lemma 3, part 2.&Lemma 2\\
Theorem 3&Theorem 4& - & -\\
 \hline
\end{tabular}
\end{center}
\end{table}

\textbf{5. Conclusions}\label{sec5}

In the paper, the two Interception Curves, one on a plane and the other one on a sphere were studied. The extensive literature about the planar curve and the mentioned connections with Lemnicate constants and Plane geometry lemmas show that it is an important and interesting topic for Mathematics and its applications. The spherical curve and its properties naturally extend these results to a geometry different from planar. The interplay between the plane and spherical geometry results and the geometric properties of the curves (5) and (7) were very helpful to find and prove these properties. Evident symmetry between the planar and spherical cases helped us to find more properties and shorter and more elegant proofs for them. Analogous to the planar case, there is a spherical pursuit curve studied in \cite{roeser} (see also \cite{loria}, page 78 of Vol. II) and it would be interesting to compare it with the Spherical Interception Curve. It would also be interesting to consider all these questions in the context of hyperbolic geometry \cite{roeser} and other planar regions \cite{morley}.

\textbf{6. Appendix}

In Section 2, we found a parametrization (5) for the curve defined by (4). It is possible to write alternative parametrizations for this curve. Let us denote $z=\frac{y}{x}$ and $z'=p\ge 1$. Then we obtain from (4), $z=\sqrt{p^2-1}-xp$. Since $dz=pdx$ and $dz=\frac{p}{\sqrt{p^2-1}}dp-pdx-xdp$, we obtain a linear differential equation $\frac{dx}{dp}+\frac{x}{2p}=\frac{p}{2\sqrt{p^2-1}}$. Noting the initial condition $x=0, y=0$, the solution of (4) can also be written as $$\left\{x(p)=\frac{1}{\sqrt{p}}\int_1^p\frac{\sqrt{t}dt}{2\sqrt{t^2-1}}, y(p)=x(p)\sqrt{p^2-1}-(x(p))^2p\ (p\ge 1)\right\}.$$

Another parametrization of the solution can be written if $z=\sqrt{p^2-1}-xp$ is solved first for $x$ to get $x=\frac{\sqrt{p^2-1}}{p}-\frac{z}{p}$ and then noting that $dx=\frac{dz}{p}$, a linear differential equation $\frac{dz}{dp}-\frac{z}{2p}=\frac{1}{2p\sqrt{p^2-1}}$ is obtained. Its solution $z(p)=\sqrt{p}\int_1^p\frac{dt}{2t\sqrt{t}\sqrt{t^2-1}}$ gives us $$\left\{x(p)=\frac{\sqrt{p^2-1}}{p}-\frac{1}{\sqrt{p}}\int_1^p\frac{dt}{2t\sqrt{t}\sqrt{t^2-1}}, y(p)=x(p)\cdot \sqrt{p}\int_1^p\frac{dt}{2t\sqrt{t}\sqrt{t^2-1}}\  (p\ge 1)\right\}.$$

\end{document}